\documentclass{amsart}
\usepackage{amsmath,amsfonts,amssymb,amsthm,latexsym,amsbsy,amstext,amsxtra,latexsym, xy}
\newtheorem{theorem}{Theorem}[section]
\newtheorem{lemma}[theorem]{Lemma}
\newtheorem{proposition}[theorem]{Proposition}
\newtheorem{corollary}[theorem]{Corollary}

\theoremstyle{definition}
\newtheorem{definition}[theorem]{Definition}
\newtheorem{example}[theorem]{Example}
\theoremstyle{remark}
\newtheorem{remark}[theorem]{Remark}
\numberwithin{equation}{section}

\def\m{\medskip}
\newcommand{\Pic}{\operatorname{Pic}}

\begin{document}

\title{Slopes of smooth curves on Fano manifolds}
\author{Jun-Muk Hwang${}^1$, Hosung Kim${}^1$, Yongnam Lee${}^2$, and Jihun Park${}^3$}

\address{Korea Institute for Advanced Study, Hoegiro 87,
         Seoul 130-722, Korea}

\email{jmhwang@kias.re.kr}

\address{Korea Institute for Advanced Study, Hoegiro 87,
         Seoul 130-722, Korea}

\email{hosung@sogang.ac.kr}

\address{Department of Mathematics, Sogang University,
         Seoul 121-742, Korea}

\email{ynlee@sogang.ac.kr}

\address{Department of Mathematics, POSTECH, Pohang, Kyungbuk,
790-784, Korea}

\email{wlog@postech.ac.kr}



\thanks{${}^1$ supported
by National Researcher Program 2010-0020413 of NRF and MEST}
\thanks{${}^2$ supported by Mid-career Researcher Program through NRF grant funded by the MEST
(No. 2010-0008752)}
\thanks{${}^3$ supported by Basic
Science Research Program through NRF grant funded by the MEST (No.
2010-0008235)}

\begin{abstract}
Ross and Thomas introduced the concept of slope stability to study
K-stability, which has conjectural relation with the existence of
constant scalar curvature  K\"ahler metric.  This paper presents a
study of slope stability of Fano manifolds of dimension $n\geq 3$
with respect to smooth curves. The question turns out to be easy
for curves of genus $\geq 1$ and the interest lies in the case of
smooth rational curves. Our main result classifies completely the
cases when a polarized Fano manifold $(X, -K_X)$ is not slope
stable with respect to a smooth curve. Our result also states that
a Fano threefold $X$ with Picard number 1 is slope stable with
respect to every smooth curve unless $X$ is the projective space.
\end{abstract}
\maketitle


\section{Introduction}
One of the fundamental problems in K\"ahler geometry is to
determine which Fano manifold $X$ admits a K\"ahler-Einstein
metric. It is expected  that the existence of a K\"ahler-Einstein
metric is closely related to the stability condition of the
polarized manifold $(X, -K_X)$. In fact, it is known that, for a
polarized manifold $(X,L)$, the existence of constant scalar
curvature  K\"ahler metric in the class $c_1(L)$ implies
K-semistability of $(X,L)$ (cf. \cite{CT08}, \cite{D05}).

The K-semistability of a polarized manifold is often very hard to
check. To remedy this,  Ross and Thomas introduced the notion of
the slope (semi-)stability of a polarized manifold $(X, L)$ and
showed that K-(semi-)stability implies slope (semi-)stability (cf.
\cite{RT07}). The question of the slope stability of a given
polarized manifold with respect to a subscheme is an interesting
algebro-geometric problem in itself. Many cases have been worked
out in \cite{PR09}, \cite{R03}, \cite{R06}, \cite{RT06},
\cite{RT07}.

An essential difficulty in this problem often lies in the
estimation of the Seshadri number of the ample line bundle along
the subvarieties.  This way, it is related to the study of
Seshadri numbers, which is an important subject in classical
algebraic geometry.

\m In this paper, we study the  slope stability of a Fano manifold
$(X, -K_X)$ with respect to smooth  curves. It is rather easy to
see that $(X, -K_X)$ is always slope stable with respect to a
smooth curve of genus $\geq 1$ (cf. Corollary~\ref{cor:high
genus}). Thus our main concern will be smooth rational curves.
Since a Fano manifold is covered by rational curves, many
geometric properties of a Fano manifold can be described by
rational curves. We  completely classify the cases when a
polarized Fano manifold $(X, -K_X)$ is not slope stable with
respect to a smooth curve. More precisely,

\begin{theorem}\label{theorem:A}
Let $X$ be a Fano manifold of dimension $n\geq 2$. If the polarized
manifold $(X,-K_X)$ is not slope stable with respect to a smooth
subvariety $Z$, then the subvariety $Z$ is a Fano manifold.
\end{theorem}

\begin{theorem}\label{theorem:B}
Let $X$ be a Fano manifold of dimension $n\geq 3$.
\begin{enumerate}
\item If the polarized manifold $(X,-K_X)$ is not slope stable
    with respect to a smooth curve $Z$, then the curve $Z$ is  one of the following
    \begin{enumerate}
    \item a rational curve whose normal bundle trivial and whose Seshadri
         constant with respect to $-K_X$ is $n$;
    \item a rational curve whose normal bunble is
    $\mathcal{O}_{\mathbb{P}^1}^{n-2}\oplus\mathcal{O}_{\mathbb{P}^1}(-1)$;
    \item a line on $\mathbb{P}^n$.
    \end{enumerate}
\item The polarized threefold $(X,-K_X)$ is slope semistable
    with respect to every smooth curve except when the curve
    $Z$ is a rational curve whose normal bundle is
    $\mathcal{O}_{\mathbb{P}^1}^{n-2}\oplus\mathcal{O}_{\mathbb{P}^1}(-1)$.
\end{enumerate}
\end{theorem}

\begin{theorem}\label{theorem:D}
Let $X$ be a Fano manifold of Picard number $1$ and dimension $n\geq
3$. Then the polarized manifold $(X,-K_X)$ is slope stable with
respect to a smooth rational curve $Z$ whose normal bundle is
trivial. Furthermore, in the case of dimension $3$, the Fano
manifold $X$ is slope stable with respect to every smooth curve $Z$
except when $X\cong\mathbb{P}^3$ and $Z$ is a line.
\end{theorem}

The proofs of the theorems make use of  the deformation of
rational curves (see \cite{CMS} and \cite{K95}), vector bundles
over manifolds whose projectivisations are Fano manifolds and the
classification of Fano threefolds with Picard number 1 (see
\cite{SH5}).

In this paper, we work over the field of complex numbers.

\m
%

\section{Slope stability and Fano bundles}
This section briefly reviews the concept of slope stability, and
proves Theorem~\ref{theorem:A}. For more details on slope
stability, we refer to \cite{RT06} and \cite{RT07}.

A polarized manifold $(X,L)$ is a pair of a smooth projective
variety $X$ with an ample line bundle $L$ on $X$.

The Seshadri constant of  a proper closed subscheme $Z$ of $X$
with respect to the ample line bundle $L$ is defined as
$$\epsilon(Z,X,L):=\max\{c\ |\ \sigma^*L-cE\ \text{ is nef }\},$$
where $\sigma:\hat{X}\rightarrow X$ is the blowup along $Z$ with
the exceptional divisor $E$. When the polarized manifold is
clearly given, we will use shortly $\epsilon(Z)$ instead of
$\epsilon(Z,X,L)$.

The following is immediate from the definition.

\begin{lemma}\label{s.Bdd}
Let $(X,L)$ be a polarized manifold and $Z$ be a a proper closed
subscheme of $X$. If there is a curve $C$ with $C\not\subset Z$
and $C\cap Z\neq\emptyset$, then $\epsilon(Z)\leq L\cdot C$.
\end{lemma}

\begin{remark}\label{r.s.Bdd}
We remark the following  facts on Seshadri constant which can be
found in Example 5.4.11
 and Proposition 5.4.15 in \cite{L04}.
\begin{itemize}
\item[(i)] Let $Z_1$, $Z_2$, $Z$ be proper closed subschemes
    of a projective variety $X$ defined by ideal sheaves
    $\mathcal I_1$, $\mathcal I_2$, $\mathcal I_1+ \mathcal
    I_2$ respectively.  Then for any ample line bundle $L$ on
    $X$, we have
$$\epsilon(Z,X,L)\geq {\rm
min}\{\epsilon(Z_1,X,L),\epsilon(Z_2,X,L)\}.$$
\item[(ii)] Consider smooth projective varieties $$Z\subset
    Y\subset X$$ and let $L$ be an ample line bundle on $X$.
    If $\epsilon(Z,X,L)<\epsilon(Y,X,L)$, then $\epsilon(Z,Y,
    L|_Y)=\epsilon(Z,X,L)$.
\end{itemize}
\end{remark}

\begin{example}\label{e.s.linear} For
any proper linear subspace $Z$ of $\mathbb P^n$,  we claim that
$\epsilon(Z,\mathbb P^{n}, -K_{\mathbb P^n}) = n+1$.

For each proper linear subspace $Z$ in $\mathbb P^n$, choose a
line $l$ in $\mathbb P^n$ so that $l\not\subset Z$ and $l\cap
Z\neq\emptyset$. Then $\epsilon(Z,\mathbb P^{n}, -K_{\mathbb
P^n})\leq (-K_{\mathbb P^n})\cdot l=n+1$ by Lemma \ref{s.Bdd}.

For the reverse inequality, we use an induction.  For any
hyperplane $H$ in $\mathbb P^n$,
\begin{align*}\epsilon(H,\mathbb P^{n}, -K_{\mathbb P^n})&={\rm max}\{x| \ -K_{\mathbb P^n}-xH \text{ is nef} \}\\
&={\rm max}\{x| \ \mathcal O_{\mathbb P^n}(n+1-x) \text{ is nef}
\}=n+1.\end{align*} Suppose the equality holds for any linear
subspace of codimension $\leq r-1$. Let $Z$ be a codimension $r$
linear subspace of $\mathbb P^n$. Choose hyperplanes
$H_1$,...,$H_r$ so that $H_1\cap ...\cap H_r=Z$. Then
$$\epsilon(Z,\mathbb P^{n}, -K_{\mathbb P^n})\geq {\rm
min}\{\epsilon(H_1\cap ...\cap H_{r-1},\mathbb P^{n}, -K_{\mathbb
P^n}),\epsilon(H_r,\mathbb P^{n}, -K_{\mathbb P^n})\}=n+1$$ by
Remark \ref{r.s.Bdd} (i) and the induction hypothesis.
\end{example}

Let $(X,L)$ be a polarized manifold of dimension $n$ with Hilbert
polynomial
$$\chi(\mathcal O_X(kL))=a_0k^n+a_1k^{n-1}+O(k^{n-2}), \ k>\!\!>0.$$
The slope of $(X,L)$ is defined by
$$\mu(X)=\mu(X,L):=\frac{a_1}{a_0}=-\frac{nK_XL^{n-1}}{2L^n}.$$
In particular, if $X$ is Fano and $L=-K_X$, then
$\mu(X)=\frac{n}{2}.$

Let $Z$ be a proper closed subscheme of $X$ and
$\sigma:\hat{X}\rightarrow X$ be the blowup along $Z$ with the
exceptional divisor $E$. For fixed $x\in \mathbb Q_{>0}$, define
$a_i(x)$ by
$$\chi(\mathcal O_{\hat X}(\sigma^*(kL)-xkE))=a_0(x)k^n+a_1(x)k^{n-1}+O(k^{n-2}), \ k>\!\!>0,\  xk\in \mathbb N.$$
Then $a_i(x)$ can be extended to all  $x\in \mathbb R$ as a
polynomial  of degree at most $n-i$. In particular, when $Z$ is a
submanifold of dimension $d$, then we have
$$a_0(x)=\frac{1}{n!}(\sigma^*L-xE)^n$$
and
$$a_1(x)=-\frac{1}{2(n-1)!}K_{\hat X}(\sigma^*L-xE)^{n-1}$$
 where  $K_{\hat
X}=\sigma^*(K_X)+(n-d-1)E$ is the canonical divisor of $\hat X$.

Set $\tilde a_i(x):=a_i-a_i(x)$. Following Definition 3.13 in
\cite{RT06}, the quotient slope of $Z$ with respect to
$\lambda\in(0,\epsilon(Z)]$ is
$$\mu_\lambda(\mathcal O_Z)=\mu_\lambda(\mathcal
O_Z,L):=\frac{\int^\lambda_o\left(\tilde{a}_1(x)+\frac{\tilde{a}_0'(x)}{2}\right)dx}{\int^\lambda_0\tilde{a_0}(x)dx}.$$

\begin{remark}\label{r.p}
This is the Remark 4.21 in \cite{RT07}. Since
$\tilde{a}_0(0)=a_0-a_0(0)=0$ and
$$\tilde{a}_0'(x)=-a_0'(x)=\frac{1}{(n-1)!}(\sigma^*L-xE)^{n-1}E> 0$$ for all $x\in
(0,\epsilon(Z))$, we have $\tilde{a}_0(x)>0$ for all $x\in
(0,\epsilon(Z))$. Therefore $\mu_{\lambda}(\mathcal O_Z)$ is
finite.
\end{remark}

\begin{definition}
We say that $(X,L)$ is slope stable (resp. slope semistable) with
respect to $Z$ if
$$\mu_\lambda(\mathcal O_Z)>(\text{ resp. }\geq\ ) \, \mu(X) \text{ for all }\lambda\in
(0,\epsilon(Z)].$$ We also say that $Z$ destabilizes (resp. strictly
destabilizes) $(X,L)$ if
$$\mu_\lambda(\mathcal O_Z) \le (\text{ resp. }<) \, \mu(X) \text{ for some }\lambda\in
(0,\epsilon(Z)].$$
\end{definition}

\begin{remark}\label{r.dstable}
Here for simplicity, we use a slightly  different definition of
slope stability from Ross-Thomas' (Definition 3.8 in \cite{RT06}):
we include  $\lambda=\epsilon(Z).$ If it is slope stable in our
sense, it is slope stable in the sense of Ross-Thomas. In
particular, the statements of our Theorems hold in both senses.
The definition of slope semistability coincides with that of
Ross-Thomas.
\end{remark}

The following theorem is the Theorem~5.1 in \cite{RT06}.
\begin{theorem}\label{cur-slope}
Let $(X,L)$ ba a polarized manifold of dimension $n\geq 3$ and
suppose that $Z$ is a smooth curve in $X$ of genus $g$ with normal
bundle $N_{Z/X}$. Then, for  $\lambda\in (0,\epsilon(Z)]$,
$$\mu_\lambda(\mathcal O_Z)=\frac{n^2(n^2-1)(L\cdot
Z)-\lambda n(n+1)[(n-2)p+2(g-1)]}{2n\lambda[(n+1)(L\cdot Z)-\lambda
p]}$$ where $p:=\deg N_{Z/X}$.
\end{theorem}
\begin{lemma}\label{s.Bdd.Cur}
Let $(X,L)$ be a polarized manifold of dimension $n\geq 2$ and
suppose that $Z$ is a smooth curve in $X$  with normal bundle
$N_{Z/X}$. Then $$(n-1)L\cdot Z-\lambda p> 0$$ for all $\lambda
\in (0,\epsilon(Z))$  where $p:=\deg N_{Z/X}$.
\end{lemma}
\begin{proof}
Let $\sigma: \tilde X\rightarrow X$ be the blowup along $Z$ with
the exceptional divisor $E$. Then the restriction
$\sigma|_E:E\cong \mathbb P(N_{Z/X}^*)\rightarrow Z$ is the
projection map of the projective normal bundle of $Z$ in $X$. Set
$\omega:=c_1(\mathcal O_{\mathbb P(N_{Z/X})}(1))$. The
Grothendieck formula $\sum_{i=0}^{n-1}\sigma^*(c_i(N_{Z/X}))\cdot
\omega ^{n-1-i}=0$ (see p.55 Remark 3.2.4 in \cite{F98}) reduces
to $\omega^{n-1}=-\sigma^*(c_1(N_{Z/X}))\cdot\omega ^{n-2}$. Since
$\omega=-E|_E$,
$$(-E)^{n-1}\cdot E=(-E|_E)^{n-1}=-\sigma^*(c_1(N_{Z/X}))\cdot(-E|_E)^{n-2}=-p.$$
 Also we have
$$(\sigma^*L)\cdot(-E)^{n-2}\cdot E=(\sigma^*L|_E)\cdot(-E|_E)^{n-2}=
(\sigma^*L|_E)\cdot c_1(\mathcal O_{\mathbb P(N_{Z/X})}(1))^{n-2}=L\cdot Z$$ and
\begin{align*}
(\sigma^*L)^i\cdot(-E)^{n-1-i}\cdot E&=(\sigma^*L|_E)^i\cdot
(-E|_E)^{n-1-i}\\&=(\sigma^*L|_E)^i\cdot c_1(\mathcal O_{\mathbb
P(N_{Z/X})}(1))^{n-1-i}=0
\end{align*}
 because $L^i\cdot Z=0$ for $i=2,...,n-1$. Therefore we
get the following equalities.
\begin{align*}
 (\sigma^*L-xE)^{n-1}\cdot E&=\sum_{i=0}^{n-1} {n-1 \choose
i}(\sigma^*L)^{i}\cdot(-xE)^{n-i-1}\cdot E\\
&=x^{n-2}\{(n-1)(\sigma^*L)\cdot(-E)^{n-2}\cdot E+x(-E)^{n-1}\cdot E\}\\
&=x^{n-2}\{(n-1)L\cdot Z-xp\}.
\end{align*}
Since $\sigma^*L-xE$ is ample for all $x\in (0,\epsilon(Z))$, we
have $(\sigma^*L-xE)^{n-1}\cdot E>0$ for all $x\in
(0,\epsilon(Z))$ and hence  we get the inequality.
\end{proof}

\begin{corollary}\label{cur-ineq}
Let $(X,L)$ be a polarized manifold of dimension $n\geq 3$ and
suppose that $Z$ is a smooth curve of genus $g$ in $X$  with
normal bundle $N_{Z/X}$. Then $Z$  destabilizes (resp. strictly
destabilizes) $(X,L)$ if and only if

$$2p\mu(X)\lambda^2-(n+1)[(n-2)p+2(g-1)+2(L\cdot
Z)\mu(X)]\lambda+n(n^2-1)(L \cdot Z) $$ $\ \leq 0\ (\ resp.\ <0)$
for some $\lambda\in (0,\epsilon(Z)]$ where $p:=\deg N_{Z/X}$.
\end{corollary}
\begin{proof}From Lemma \ref{s.Bdd.Cur}, we know that the
denominator of
 $\mu_\lambda(\mathcal O_Z)$ in Theorem \ref{cur-slope} is  positive for all $\lambda\in (0,
 \epsilon(Z)]$. Thus  the  corollary comes from the definition.
\end{proof}

The following lemma is Remark in Section 8 in \cite{R03}. We give
the proof for the readers' convenience.

\begin{lemma}\label{r.fano} Let $(X,-K_X)$ be a
Fano manifold of dimension $n$ and $Z$ be  a smooth closed
subscheme of codimension $r$.  If $\epsilon(Z)\leq r$, then
$(X,-K_X)$ is slope stable with respect to $Z$.
\end{lemma}
\begin{proof} If $\epsilon(Z)\leq r$,
\begin{align*}&-\mu(X)\tilde{a}_0(x)+\tilde{a}_1(x)+\frac{1}{2}\tilde{a}_0'(x)\\
&=-\mu(X)(a_0-a_0(x))+(a_1-a_1(x))-\frac{1}{2}a_0'(x)\\
&=-\mu(X)a_0+a_1+\mu(X)a_0(x)-a_1(x)-\frac{1}{2}a_0'(x)\\
&=\frac{n}{2}\frac{(\sigma^*(-K_X)-xE)^n}{n!}+\frac{K_{\hat
X}(\sigma^*(-K_X)-xE)^{n-1}}{2(n-1)!}+\frac{(\sigma^*(-K_X)-xE)^{n-1}E}{2(n-1)!}\\
&=\frac{1}{2(n-1)!}(\sigma^*(-K_X)-xE)^{n-1}(\sigma^*(-K_X)-xE+K_{\hat{X}}+E)\\
&=\frac{1}{2(n-1)!}(\sigma^*(-K_X)-xE)^{n-1}(\sigma^*(-K_X)-xE+\sigma^*(K_X)+(r-1)E+E)\\
&=\frac{1}{2(n-1)!}(r-x)(\sigma^*(-K_X)-xE)^{n-1}E>0
\end{align*}
for all $x\in (0,\epsilon(Z))$. Since  $\tilde{a}_0(x)>0$  for all
$x\in (0,\epsilon(Z))$ by  Remark \ref{r.p}, we have
$\mu_{\lambda}(\mathcal O_Z)>\mu(X)$ for all
$\lambda\in(0,\epsilon(Z)]$.
\end{proof}

\begin{lemma}\label{lemma:Fano}
Let $Z$ be a smooth subvariety of a Fano manifold $X$ with
codimension $r$. Let $\pi : Y\to X$ be the blow-up along the
subvariety $Z$ with the exceptional divisor $E$. If the Seshadri
constant $\epsilon(Z,X, -K_X)$ is  strictly bigger than $r$, then
the exceptional divisor $E$ and subvariety $Z$ must be Fano
manifolds.
\end{lemma}
\begin{proof}
If the projective bundle $E$ over $Z$ is a Fano manifold, then $Z$
is a Fano manifold(Theorem~1.6 in \cite{SzuWi90}). Therefore, it
is enough to show that the exceptional divisor $E$ is a Fano
manifold. We have
\[-(K_Y+E)=\sigma^*(-K_X)-rE.\]
Since $\epsilon(Z,X, -K_X)>r$, the divisor $-(K_Y+E)$ is ample. By
adjunction, we see that the anticanonical divisor $-K_E$ on $E$ is
ample.
\end{proof}

\begin{proof}[Proof of Theorem~{\rm\ref{theorem:A}}]
Lemma~\ref{r.fano} implies that $\epsilon(Z,X, -K_X)$ is strictly
bigger than $r$. Then Theorem~\ref{theorem:A} immediately follows
from Lemma~\ref{lemma:Fano}.
\end{proof}

\begin{corollary}\label{cor:high genus}
The polarized manifold $(X, -K_X)$ is slope stable with respect to
every non-rational smooth curve.
\end{corollary}
\begin{proof}
It immediately follows from Theorem~\ref{theorem:A}.
\end{proof}

\begin{lemma}\label{lemma:Fano-Bundle}
If the projective space bundle
 over $\mathbb{P}^1$
\[V:=\mathbb{P}\left(\bigoplus_{i=1}^{n}\mathcal{O}_{\mathbb{P}^1}(a_i)\right),
\ \ \ a_1=0, \ a_i\in\mathbb{Z}_{\geq 0}.\] is a Fano manifold,
then $V$
 is either $\mathbb P(\mathcal{O}_{\mathbb{P}^1}^{n})$ or $\mathbb
P(\mathcal{O}_{\mathbb{P}^1}^{n-1}\oplus
\mathcal{O}_{\mathbb{P}^1}(1))$.
\end{lemma}
\begin{proof}
 Let
$\pi:V\to\mathbb{P}^1$ be the natural projection. Let $F$ be a
fiber of $\pi$ and $M$ be a divisor given by the Grothendieck
tautological invertible sheaf. Then $\Pic
(V)=\mathbb{Z}M\oplus\mathbb{Z}F$ and
$$-K_V=nM+(2-d)F,$$ where $d=\sum a_i$. Note that $F^2\cdot
M^{n-2}=0$, $F\cdot M^{n-1}=1$ and $M^n=d$.

Consider the section $s$ of $\pi$ that corresponds to the quotient
$\mathcal{O}_{\mathbb{P}^1}$ of the bundle
$\bigoplus_{i=1}^{n}\mathcal{O}_{\mathbb{P}^1}(a_i)$. Let
$l=s(\mathbb{P}^1)$. We then see
\[l\equiv M^{n-1}-dF\cdot M^{n-2}.\]
Since $-K_V$ is ample, we obtain
\[-K_V\cdot l=(nM+(2-d)F)\cdot (M^{n-1}-dF\cdot
M^{n-2})=2-d>0.\] The inequality $d<2$ completes the proof.
\end{proof}

\begin{corollary}\label{cor:normal bundle}
If the polarized manifold $(X, -K_X)$ is not slope stable with
respect to a smooth rational  curve $Z$, then the normal bundle
$N_{Z/X}$ is one of the following;
\begin{itemize}
\item[$\bullet$]
    $N_{Z/X}=\mathcal{O}_{\mathbb{P}^1}^{n-1}(-a)$;
\item[$\bullet$]
    $N_{Z/X}=\mathcal{O}_{\mathbb{P}^1}^{n-2}(-a)\oplus
    \mathcal{O}_{\mathbb{P}^1}(-a-1),$
\end{itemize}
where $a$ is an integer.
\end{corollary}
\begin{proof}
Let $Z$ be a smooth rational curve on $X$. Suppose that the
polarized manifold $(X, -K_X)$ is not slope stable with respect to
the smooth rational curve $Z$.  Let $\pi : Y\to X$ be the blowup
along the curve $Z$ with the exceptional divisor $E$. Then
Lemma~\ref{lemma:Fano} implies that the exceptional divisor $E$ is
a smooth Fano manifold. The exceptional divisor $E$ is isomorphic
to $\mathbb{P}(N_{Z/X}^*)$, where $N_{Z/X}$ is the normal bundle
of $Z$ on $X$. The bundle $N_{Z/X}^*$ can be decomposed into
$\bigoplus_{i=1}^{n-1}\mathcal{O}_{\mathbb{P}^1}(a_i)$, $a_1\leq
a_2\leq\cdots\leq a_{n-1}$. Note that
\[\mathbb{P}(N_{Z/X}^*)\cong
\mathbb{P}\left(\bigoplus_{i=1}^{n-1}\mathcal{O}_{\mathbb{P}^1}(a_i-a_1)\right).\]
By Lemma~\ref{lemma:Fano-Bundle} we obtain either
\begin{itemize}
\item[$\bullet$] $a_i=a_1$ for each $i$ or \item[$\bullet$]
    $a_i=a_1$ for $1\leq i\leq n-2$ and $a_{n-1}=a_1+1$.
\end{itemize}

\end{proof}

Corollary~\ref{cor:normal bundle} will be used to prove Theorems
1.2, 1.3, and 1.4. A vector bundle over a manifold whose
projectivisation is a Fano manifold is called a Fano bundle.
Studies on Fano bundles may show a way to understand higher
dimensional submanifolds destablizing Fano manifolds.

\m

\section{Slopes of smooth rational curves in a Fano manifold}

Let $X$ be a Fano manifold of dimension $n\geq 2$. Throughout we
will fix the polarization given by the anticanonical bundle
$-K_X$. Let $Z$ be a smooth rational curve  in $X$ with the normal
bundle $N_{Z/X}$.

\begin{proposition}\label{s.Bdd.Fano}
Let $X$ be a Fano manifold of dimension $n\geq 2$ and $Z$ be a
smooth rational curve  in $X$.
\begin{enumerate}
\item[(i)]If $(-K_X)\cdot Z=2$, then $(X,-K_X)$ is slope
    stable (resp. slope semistable) with respect to $Z$ if and
    only if $\epsilon(Z)< n$(resp. $\leq n$).
\item[(ii)]If $(-K_X)\cdot Z=1$, then  $(X,-K_X)$ is slope
    stable (resp. slope semistable) with respect to $Z$ if and
    only if $\epsilon(Z)< \sqrt{n^2-1}$ (resp. $\leq
    \sqrt{n^2-1}$).
\item[(iii)] If $(-K_X)\cdot Z\geq 3$ and $\epsilon(Z)\leq
    (-K_X)\cdot Z$, then $(X,-K_X)$ is slope semistable with
    respect to $Z$, and $Z$ destabilizes $(X,-K_X)$ if and
    only if $\epsilon(Z)=(-K_X)\cdot Z=n+1$.
\end{enumerate}
\end{proposition}

\begin{proof}
 Set $p:=c_1(N_{Z/X})$. Then $(-K_X)\cdot Z=p+2$.  Note $\mu(X)=\frac{n}{2}$. From Corollary \ref{cur-ineq}, we get
that  $Z$ destabilizes (resp. strictly destabilizes) $(X,-K_X)$ if
and only if
\begin{align*}
f(\lambda)&:=pn\lambda^2-2(n^2-1)(p+1)\lambda+n(n^2-1)(p+2)\leq0\
({\text\ resp\ }. <0)
\end{align*}
 for some  $\lambda\in (0,\epsilon(Z)]$.
 If $(-K_X)\cdot Z=2$, $p=0$ and hence
    $f(\lambda)=-2(n^2-1)(\lambda-n)$, which implies (i).
If $(-K_X)\cdot Z=1$, $p=-1$ and so
    $f(\lambda)=-n(\lambda ^2-(n^2-1))$, which implies (ii).
 Assume $(-K_X)\cdot Z\geq 3$. Then $p\geq 1$, and hence
\begin{align*}
f(\lambda)
&=pn\left(\lambda-\frac{(p+1)(n^2-1)}{pn}\right)^2-\frac{(p+1)^2(n^2-1)^2}{pn}+n(n^2-1)(p+2)\\
&=pn\left(\lambda-\frac{(p+1)(n^2-1)}{pn}\right)^2+\frac{(n^2-1)((p+1)^2-n^2)}{pn}
\end{align*}
Thus if $p>n-1$, then $f(\lambda)>0$ for all $\lambda$ and $(X,
-K_X)$ is slope stable with respect to $Z$. It remains to consider
the case  $1 \leq p\leq n-1$. Note that $f$ has a minimum value at
$\frac{(p+1)(n^2-1)}{pn}$ and
\begin{align*}&(p+2)-\frac{(p+1)(n^2-1)}{pn} \\&\leq
(n+1)-\frac{(p+1)(n^2-1)}{pn}=\frac{(n+1)(p-n+1)}{pn}\leq 0
\end{align*}
and hence  $p+2\leq \frac{(p+1)(n^2-1)}{pn}$.
 Since
\begin{align*}
f(p+2)&=pn(p+2)^2-2(n^2-1)(p+1)(p+2)+n(n^2-1)(p+2)\\
&=(p+2)\{pn(p+2)-2(n^2-1)(p+1)+n(n^2-1)\}\\
&=(p+2)(p-n+1)\{n(p-n+1)+2\},
\end{align*}
$f(p+2)\geq 0$ and  $f(p+2)=0$ if and only if $p=n-1$. This means
that
    $f(\lambda)\geq 0$ for all $\lambda\in (0, p+2]$,
    and $f(\lambda)=0$ for some $\lambda\in (0, p+2]$ if and only if $p=n-1$ and $\lambda=p+2$.

Hence if $\epsilon(Z)\leq (-K_X)\cdot Z=p+2$, $f(\lambda)\geq 0$
for all $\lambda\in (0, \epsilon(Z)]$, and $f(\lambda)=0$ for some
$\lambda\in (0, \epsilon(Z)]$ if and only if $\epsilon(Z)=p+2=n+1$
and $\lambda=n+1$. We are done.
\end{proof}

\begin{lemma}\label{BSCF} Let $X$ be a Fano manifold of dimension $\geq 2$.
Then for every  smooth rational curve $Z$ with $(-K_X)\cdot Z\geq
3$, we have $\epsilon(Z)\leq (-K_X)\cdot Z$.
\end{lemma}
\begin{proof}
Let $h:Z\hookrightarrow X$ be the embedding of $Z$ in $X$. By
II.2.2 in \cite{K95}, the deformation space of $h$ fixing one
point has dimension $\geq (-K_X)\cdot Z\geq 3$. Since  the group
of automorphisms of $\mathbb P^1$ fixing one point has dimension
2, $Z$ moves in $X$ with one point fixed. So one can find a curve
$C(\neq Z)$ in $X$ which is algebraically equivalent to $Z$ and
$C\cap Z\neq \emptyset$. Hence $\epsilon(Z)\leq (-K_X)\cdot
C=(-K_X)\cdot Z$ by Lemma \ref{s.Bdd}.
\end{proof}

\begin{proposition}\label{PF2} Let $Z \subset X$ be a smooth rational curve on a Fano
manifold of dimension $n$  with $(-K_X) \cdot Z = n+1.$ Suppose
that $(-K_X) \cdot C \geq n+1$ for any rational curve $C \subset
X$ with $Z \cap C \neq \emptyset.$ Then $X \cong \mathbb P^n$ and
$Z$ is a line on $\mathbb P^n$. \end{proposition}

\begin{proof}
Using the 11th condition of Corollary 0.4 in  \cite{CMS}, it
suffices to show that for a general point $x_0 \in X$, every
rational curve $R$ through $x_0$ satisfies $(-K_X) \cdot R \geq
n+1$. Suppose not. Then rational curves of $(-K_X)$-degree $\leq
n$ cover the whole variety $X$. In particular, there exists a
rational curve $C$ with $(-K_X) \cdot C \leq n$ passing through a
point of $Z$, a contradiction. \end{proof}

\begin{proposition}\label{prop:3}
Let $X$ be a Fano manifold of dimension $n\geq 2$ with the
anticanonical polarization. Let $ Z \subset X$ be a smooth
rational curve. If $(-K_X) \cdot Z \geq 3$, then $X$ is slope
stable with respect to $Z$, except when $X \cong \mathbb P^n$ and
$Z$ is a line;
\end{proposition}

\begin{proof} Lemma \ref{BSCF} says that $\epsilon(Z)\leq (-K_Z)\cdot Z$ and
hence, by Proposition \ref{s.Bdd.Fano} (iii), $(X,-K_X)$ is slope
stable with respect to $Z$ except when $\epsilon(Z)=(-K_X)\cdot
Z=n+1$.

If $\epsilon(Z)=(-K_X)\cdot Z=n+1$,   $(-K_X) \cdot C \geq n+1$
for any rational curve $C (\neq Z)\subset X$ with $Z \cap C \neq
\emptyset$ by Lemma \ref{s.Bdd}. Proposition \ref{PF2} says that
this condition is equivalent to $X\cong \mathbb P^n$ and $Z$ is a
line in $\mathbb P^n$.
\end{proof}

\begin{remark}
Let $l$ be a line in $\mathbb P^n$ ($n\geq 2$). Then $(-K_{\mathbb
P^n})\cdot l=n+1$ and  $\epsilon(l,-K_{\mathbb P^n})=n+1$ by
Example \ref{e.s.linear}. Thus this is an example satisfying the
upper bound of Seshadri constant in Lemma \ref{BSCF}. In fact,
$(\mathbb P^n, -K_{\mathbb P^n})$ is slope semistable but not
slope stable with respect to $l$. Proposition \ref{prop:3} says
that $\mathbb P^n$ is the only Fano manifold which has a
destabilizing smooth rational curve with anticanonical degree
$\geq 3$.
\end{remark}

\begin{proposition}\label{P1} Let $X$ be a Fano manifold of dimension $n$.
Assume that there exists a smooth rational curve $Z \subset X$
with trivial normal bundle. Then there exists a rational curve $C$
with $C \neq Z$, $C \cap Z \neq \emptyset$ and $(-K_X )\cdot C
\leq n$. \end{proposition}

\begin{proof}
The  deformations of $Z$ form an $(n-1)$-dimensional family of
smooth rational curves with trivial normal bundles. It follows
that there exists an irreducible complete subscheme $Y$ in the
Hilbert scheme of curves on $X$ with a universal family $\pi: U
\to Y$ and a dominant  morphism $\rho: U \to X$ with the following
properties.
\begin{enumerate} \item[(i)] There exists a dense open subset $Y^o
\subset Y$ such that $\rho|_{U^o}: U^o \to X$ is unramified
where $U^o := \pi^{-1}(Y^o)$ and $\pi|_{U^o}$ is a $\mathbb
P^1$-bundle.
\item[(ii)] For each $y \in Y$, let $\ell_y \subset X$ be the
    curve $\rho(\pi^{-1}(y))$. Then $Z= \ell_{y_o}$ for some
    $y_o \in Y$ and, for each $y \in Y^o$,
    $\rho|_{\pi^{-1}(y)} : \mathbb  P^1 \to \ell_{y}$ is an
    embedding.
    \item[(iii)] For $y_1 \in Y^o$ and $y_2 \in Y\backslash
        \{y_1\}$, the curve $\ell_{y_1}$ is distinct from any
        irreducible component of $\ell_{y_2}$.
\end{enumerate}

Suppose that the morphism $\rho$ is not birational. Then, by (i)
and (ii), for a general point $x \in Z$, we have two $u_1 \neq u_2
\in \rho^{-1}(x)$ such that $\pi(u_1) = y_o \neq \pi(u_2)$. By the
property (iii), a component of $\ell_{\pi(u_2)}$ through $x$ gives
the desired curve $C$. Suppose that the morphism $\rho$ is
birational. Then, by (i), we may regard $U^o$ as an open subset in
$X$. Now apply Theorem 2.1 of \cite{KMM} to get a rational curve
$C_y$ intersecting $\ell_y$ for a general $y \in Y^o$ such that
$C_y \neq \ell_y$ and $(-K_X) \cdot C_y \leq n+1$. The proper
transform $\tilde{C}_y$ in $U$ is a rational curve with $\dim
\pi(\tilde{C}_y) = 1$. From the generality of $y$, we can find an
irreducible curve $\tilde{C}_{y_o}$ intersecting $\pi^{-1}(y_o)$
with $\dim \pi(\tilde{C}_{y_o}) =1$ such that the image $C=
\rho(\tilde{C}_{y_o})$ satisfies $C \neq Z$, $C \cap Z \neq
\emptyset$ and $(-K_X) \cdot C \leq n+1$. We are done if $(-K_X)
\cdot C < n+1$. Assume that $(-K_X) \cdot C = n+1$. We can deform
$C$ with a point $x_o= C \cap Z$ fixed to have an
$(n-1)$-dimensional family of distinct rational curves through
$x_o$. If this family of rational curves through $x_o$ form a
complete family, then $X \cong \mathbb  P^n$ by \cite{CMS} and
cannot contain a rational curve with trivial normal bundle unless
$n=1$. Thus this family of curves cannot be complete and there
exists a reducible (or non-reduced) curve $C'$ with $(-K_X) \cdot
C' = n+1$. If $C'$ has more than one component, then one of the
component, say $C^{"}$, must be distinct from $Z$, intersects $Z$
and satisfies $(-K_X) \cdot C^{"} <n+1$. Thus we are done. If on
the other hand, $C'$ is irreducible and non-reduced, its reduction
$C_0$ is a rational curve with $(-K_X) \cdot C_0 <n+1$ passing
through $x_o$. If $C_0 \neq Z$, we are done. It remains to exclude
the possibility of $C_0 = Z$. If this happens, it means that the
curve $C$ is numerically equivalent to a multiple of $Z$.   Take
an irreducible hypersurface $D \subset Y$ such that
\begin{enumerate} \item[(a)] $y_o \not\in D$, \item[(b)] $D \cap
\pi(U^o \cap \tilde{C}) \neq \emptyset$ and \item[(c)]
$\pi(U^o \cap \tilde{C}) \not\subset D$. \end{enumerate} Then
$C$ has non-empty intersection with the hypersurface
$\rho\circ\pi^{-1}(D)$, but $Z$ is disjoint from
$\rho\circ\pi^{-1}(D)$ from the birationality of $\rho$. This
is a contradiction to the assumption that $C$ is in the same
numerical class as some multiple of $Z$.\end{proof}

\begin{proof}[Proof of Theorem~{\rm \ref{theorem:B}}]Suppose that the polarized Fano
manifold $(X,-K_X)$ is not slope stable with respect to a smooth
curve $Z$. By Corollary~\ref{cor:normal bundle} the curve $Z$ is a
rational curve whose normal bundle is either
\begin{itemize}
\item[$\bullet$]
    $N_{Z/X}=\mathcal{O}_{\mathbb{P}^1}^{n-1}(-a)$ or
\item[$\bullet$]
    $N_{Z/X}=\mathcal{O}_{\mathbb{P}^1}^{n-2}(-a)\oplus
    \mathcal{O}_{\mathbb{P}^1}(-a-1),$
\end{itemize}
where $a$ is an integer.

Now we suppose that if the Fano manifold $X$ is $\mathbb{P}^n$,
then $Z$ is not a line. By Proposition~\ref{prop:3} we see that
$-K_X\cdot Z$ is either $2$ or $1$. Therefore we obtain
\[\deg (N_{Z/X})=-2-K_X\cdot Z= 0 \ \ \ \mbox{ if } -K_X\cdot Z=2; \]
\[\deg (N_{Z/X})=-2-K_X\cdot Z= -1 \ \ \ \mbox{ if } -K_X\cdot Z=1. \]
If the normal bundle of $Z$ is
$\mathcal{O}_{\mathbb{P}^1}^{n-1}(-a)$, then $\deg
(N_{Z/X})=-(n-1)a$. Therefore, $-K_X\cdot Z=2$ and $a=0$. If the
normal bundle of $Z$ is
$\mathcal{O}_{\mathbb{P}^1}^{n-2}(-a)\oplus
\mathcal{O}_{\mathbb{P}^1}(-a-1)$, then $\deg
(N_{Z/X})=-(n-1)a-1$. Then $a=0$. Consequently, the curve $Z$ is a
rational curve with trivial normal bundle or with
$N_{Z/X}=\mathcal{O}_{\mathbb{P}^1}^{n-2}\oplus
    \mathcal{O}_{\mathbb{P}^1}(-1)$.

If $Z$ is a rational curve with trivial normal bundle then
Proposition~\ref{P1} shows that we have a rational curve $C$ with
$C\ne Z$, $Z\cap C\ne\emptyset$ and $-K_X\cdot C\leq n$. It then
follows from Lemma~\ref{s.Bdd} that $\epsilon(X,-K_X, Z)\leq
-K_X\cdot C\leq n$. Then Proposition~\ref{s.Bdd.Fano} completes
the proof.
\end{proof}

\begin{remark}\label{PD}
Let $(X_i,L_i)$, $i=1,2$, be polarized manifolds. By Remark 3.9 in
\cite{R06}
$$\mu(X_1\times X_2, L_1\boxtimes
L_2)=\mu(X_1,L_1)+\mu(X_2,L_2)$$ and,  for  a subscheme $Z$ of
$X_2$, we have
$$\mu_{\lambda}(\mathcal O_{X_1\times Z},L_1\boxtimes L_2)=\mu(X_1,L_1)+\mu_{\lambda}(\mathcal
O_Z,L_2)$$ and
$$\epsilon(X_1\times Z,L_1\boxtimes L_2)=\epsilon(Z,L_2).$$

So  $Z$ destabilizes (resp. strictly destabilizes) $(X_2,L_2)$ if
and only if $X_1\times Z$ destabilizes (resp. strictly
destabilizes) $(X_1\times X_2, L_1\boxtimes L_2)$. Since any
polarized manifold is slope semistable with  respect to any smooth
point by Theorem 4.29 in \cite{RT07}, $(X_1\times X_2,
L_1\boxtimes L_2)$ is slope semistable with respect to any fiber
$X_1\times \{p\}$, $p\in X_2$.
\end{remark}

\begin{example}\label{PDPP}
Let $X:=\mathbb P^1\times \mathbb P^{n-1}$ and $p\in \mathbb
P^{n-1}$. Since $\epsilon(\mathbb P^1\times
\{p\},X,-K_X)=\epsilon(p,,\mathbb P^{n-1},-K_{\mathbb
P^{n-1}})=n$, $\mathbb P^1\times \{p\}$ destabilizes $(X,-K_X)$ by
Proposition \ref{s.Bdd.Fano} (i) although $(X,-K_X)$ is slope
semistable with respect to $ \mathbb P^1\times \{p\}$ by Remark
\ref{PD}. This shows that we cannot improve Theorem
\ref{theorem:B} to slope stability.
\end{example}

\begin{remark}
We remark that Theorem 4.29 in \cite{RT07} says that
$\epsilon(p,X,-K_X)\leq n+1$ for any smooth point $p\in X$. We
also know that $\epsilon(p,\mathbb P^{n}, -K_{\mathbb P^n})=n+1$
for any smooth point $p$ in $\mathbb P^n$. The following Lemma
\ref{SPF} shows that $\mathbb P^n$ is the unique Fano manifold
having a point satisfying the equality.
\end{remark}

\begin{lemma}\label{SPF} Let $X$ be a Fano manifold of dimension $n\geq 3$
with $X\not\cong\mathbb P^n$. Then  $\epsilon(p, X, -K_X)\leq n$
for all point $p\in X$.\end{lemma}
\begin{proof}If $X\neq \mathbb P^n$,  rational curves of $(-K_X)$-degree
$\leq n$ cover the whole variety $X$ by the 11th condition of
Corollary 0.4 in  \cite{CMS}. Hence, for any point $p\in X$, there
exists a rational curve $C'$ with $(-K_X) \cdot C' \leq n$ passing
through it and so $\epsilon(p, X, -K_X)\leq n$.\end{proof}

\begin{proposition}Let $X=\mathbb P^1\times  Y$ where $Y$ is a Fano manifold of
dimension $n-1$. Then there exists a point $p\in Y$ such that the
fiber $\mathbb P^1\times \{p\}$  destabilizes $(X, -K_X)$ if and
only if $Y\cong \mathbb P^{n-1}$.\end{proposition}

\begin{proof}If $Y\not\cong \mathbb P^{n-1}$, $\epsilon(p,Y,-K_Y)\leq n-1$
for any point $p$ in $Y$ by Lemma \ref{SPF}. So $\epsilon(\mathbb
P^1\times \{p\},X,-K_X)\leq n-1$ by  Remark \ref{PD}. Since the
normal bundle of $\mathbb P^1\times \{p\}$ is trivial, by
Proposition \ref{s.Bdd.Fano} (i), $(X,-K_X)$ is slope stable with
respect to $\mathbb P^1\times \{p\}$ for any point $p\in Y$.

If $Y\cong \mathbb P^{n-1}$, then  any fiber $\mathbb P^1\times
\{p\}$ destabilizes $(X,-K_X)$ by Example \ref{PDPP}.
\end{proof}

\begin{lemma}\label{EL1}
Let $X$ be a smooth Fano threefold and let $C$ be a smooth curve
on $X$. Let $\sigma :\tilde{X}\rightarrow X$ be the blowup of $X$
along the curve $C$. Denote the exceptional divisor of $\sigma$ by
$E$. Then $(-K_{\tilde X})^2\cdot E=(-K_X)\cdot C+2-2g(C)$.
\end{lemma}
\begin{proof}
From Lemma 2.2.14. in \cite{SH5} or the proof of Lemma
\ref{s.Bdd.Cur}, we have the following equalities.
\begin{align*}
(-K_{\tilde X})^2E&=(\sigma^*K_X+E)^2E\\
&=(\sigma^*K_X)^2E+2\sigma^*K_X\cdot E^2+E^3\\
&=0-2K_X\cdot C-\deg(N_{C/X})\\
&=-K_X\cdot C-2g(C)+2.
\end{align*}
\end{proof}

\begin{remark}
Let $\pi: X\rightarrow \mathbb P^3$ be the blowup along a line
$l$, and let $E$ be the exceptional divisor of $\pi$. Then
$E=\mathbb F_0$. Let $Z$ be a fiber of the map $\pi|_E: E\to l$.
We note that $N_{Z/X}=\mathcal O_{\mathbb P^1}\oplus \mathcal
O_{\mathbb P^1}(-1)$. Now we want to show that
$\epsilon(Z,X,-K_X)=3$.

Let $C$ be the section of $\pi|_E: E\to l$ with $C^2=0$. Then
 $-K_X|_E=C+\alpha Z$ for some $\alpha\in
\mathbb Z$. Since $(-K_X|_E)^2=(-K_{\mathbb P^3})\cdot l+2=6$ by
Lemma \ref{EL1} and $(-K_X|_E)^2=(C+\alpha Z)^2=2\alpha=6$, we
have $\alpha=3$ and hence $-K_X|_E=C+3Z$. Therefore
$\epsilon(Z,X,-K_X)\leq (-K_X)\cdot C=(-K_X|_E)\cdot C=(C+3Z)\cdot
C=3$ by Lemma \ref{s.Bdd}.

To show the equality, suppose not, i.e. $\epsilon(Z,X,-K_X)<3$.
From Example \ref{e.s.linear}, we have $\epsilon(l,\mathbb
P^3,-K_{\mathbb P^3})=4$ and so
\begin{align*}
\epsilon(E,X,-K_X)&=\epsilon(E,X,-\pi^*(K_{\mathbb
P^3})-E)\\&={\rm max}\{x| \ (-\pi^*(K_{\mathbb P^3})-E)-xE  \text{ is nef}\}\\
&={\rm max}\{x| \ -\pi^*(K_{\mathbb P^3})-(x+1)E  \text{ is nef}\}\\
&={\rm max}\{x| \ -\pi^*(K_{\mathbb P^3})-xE  \text{ is nef}\}-1\\
&= \epsilon(l,\mathbb P^3,-K_{\mathbb P^3})-1=3.
\end{align*}
Therefore the assumption $\epsilon(Z,X,-K_X)<3$ implies  that
$$\epsilon(Z,X,-K_X)<\epsilon(E,X,-K_X),$$  and hence
$\epsilon(Z,X,-K_X)=\epsilon(Z,E,-K_X|_E)$ by the Remark
\ref{r.s.Bdd} (ii). But we have
$$\epsilon(Z,E,-K_X|_E)={\rm max}\{x\ |\ (-K_X|_E)-xZ=C+(3-x)Z \
\text{is nef}\}=3,$$ a contradiction.

In conclusion, we showed that $\epsilon(Z,X,-K_X)=3$. Therefore,
$Z$ strictly destabilizes $(X, -K_X)$ by Proposition
\ref{s.Bdd.Fano} (ii), and we cannot improve
Theorem~\ref{theorem:B}. We however note that the Picard number of
$X$ is $2$.
\end{remark}

\begin{proof}[Proof of Theorem~{\rm \ref{theorem:D}}] Let  $Z$ be a
smooth rational curve on $X$.  By Theorem~\ref{theorem:B}, we only
need to consider a smooth rational curve with normal bundle either
trivial or $\mathcal O_{\mathbb
    P^1}\oplus\mathcal O_{\mathbb P^1}(-1)$ when $n=3$.

If the normal bundle $N_{Z/X}$ is trivial,  there is a curve
$C(\neq Z)$
    which is algebraically equivalent to $Z$ and $C\cap Z\neq
    \emptyset$ by Proposition 2 in  \cite{H03}.
    So $\epsilon(Z)\leq (-K_X)\cdot C=(-K_X)\cdot Z=2$ by Lemma \ref{s.Bdd}.
    Therefore the polarized manifold  $(X,-K_X)$ is slope stable with respect to $Z$
    by Proposition \ref{s.Bdd.Fano} (i).

Therefore, we have only to consider a smooth rational curve $Z$ on
a Fano threefold $X$ with the normal bundle $N_{Z/X}\cong\mathcal
O_{\mathbb P^1}\oplus\mathcal O_{\mathbb P^1}(-1)$.

Since $(-K_X)\cdot Z=\deg(N_{Z/X})+2=1$, the Fano manifold $X$ is
of Fano index $1$. Let $g$ be the genus of $X$, i.e.,
$g=\frac{1}{2}(-K_X)^3+1$.
\begin{enumerate}
\item[\textit{Case}] $g\geq 4$: By Remark 4.3.6 in \cite{SH5},
    there exists a curve $C(\neq Z)$ such that $(-K_X)\cdot
    C=1$ and $C\cap Z\neq \emptyset$. So $\epsilon(Z)\leq
    (-K_X)\cdot C=1$ by Lemma \ref{s.Bdd}. Hence the result
    follows from Proposition \ref{s.Bdd.Fano} (ii).
\item[\textit{Case}] $g=3$: Propositions 4.1.11 and 4.1.12 in
    \cite{SH5} show that the Fano manifold $X$ belongs to one
    of the following  two cases.

\begin{itemize}
\item[(1)] $X$ is anticanonically embedded in $\mathbb
    P^4$ as a quartic  and $Z$ is a line.

  Choose a plane $H $ so that $H\cap X=Z\cup Z'$ and $Z$
is not a component of $Z'$. Set $C:=Z'$. Then $(-K_X)\cdot
C=3$.

Let $\sigma :\tilde{X}\rightarrow X$ be the blowing up of
$X$ with center $Z$, and  $E$ be the exceptional divisor
of $\sigma$. Let $\bar{C}$ be the proper transform of $C$.
Since the intersection number of $C$ and $Z$ in $H$ is 3,
$E\cdot \bar{C}=3$. Hence $(\sigma^*(-K_X)-xE)\cdot
\bar{C}=(-K_X)\cdot C-xE\cdot\bar{C}=3-3x$. So if
$\sigma^*(-K_X)-xE$ is nef, then $x\leq 1$. Therefore
$\epsilon(Z)\leq 1$ and   hence   $(X,-K_X)$ is slope
stable with respect to $Z$ by Proposition \ref{s.Bdd.Fano}
(ii).

\item[(2)] The anticanonical linear system $|-K_X|$
    defines a finite morphism
    $\varphi=\varphi_{|-K_X|}:X\rightarrow Q\subset
    \mathbb P^4$ of degree 2 onto a nonsingular quadric
    $Q$ ramified along a surface $S$ of degree 8.

 Set  $Z':=\varphi(Z)$. Then $Z'$ is a line in $\mathbb
 P^4$. Choose a line  $l(\neq Z')$  in $Q$ with $l\cap
 Z'\neq \emptyset$. Let $C$ be a component of
$\varphi^{-1}(l)$ intersecting $Z$. Then $(-K_X)\cdot
C\leq 2\deg l= 2$ and $C\cap Z\neq \emptyset$. So
$\epsilon(Z)\leq (-K_X)\cdot C\leq 2$ by Lemma
\ref{s.Bdd}. Therefore $(X,-K_X)$ is slope stable with
    respect to $Z$ by Proposition \ref{s.Bdd.Fano}
    (ii).\end{itemize}
\item[{\textit Case}] $g=2$: The anticanonical linear system
    $|-K_X|$ defines a finite morphism
    $\varphi=\varphi_{|-K_X|}:X\rightarrow  \mathbb P^3$ of
    degree 2 ramified along a surface of degree 6 (Proposition
    4.1.11 in \cite{SH5}). Let $Z':=\varphi(Z)$. Choose a line
    $l$ on $\mathbb P^3$ such that $l\neq Z'$ and $l\cap
    Z'\neq \emptyset$. Take an irreducible component $C$ of
    $\varphi^{-1}(l)$ intersecting $Z$.  Then $-K_X\cdot C\leq
    2$ and $Z\cap C\neq \emptyset$. So $\epsilon(Z)\leq
    (-K_X)\cdot C\leq2$ by Lemma \ref{s.Bdd}. Therefore
    $(X,-K_X)$ is slope stable with respect to $Z$ by
    Proposition \ref{s.Bdd.Fano} (ii).
\end{enumerate}
This completes the proof. \end{proof} As seen in the proof of
Theorem~\ref{theorem:D}, on a given Fano manifold, its Fano index
gives us a lower bound for the intersection numbers of its
anticanonical divisor with curves. Note that the Fano index of an
$n$-dimensional Fano manifold is at most $n+1$. Furthermore, the
theorem of Kobayashi-Ochiai (see \cite{KoOch73}) states that it is
$n+1$ (resp. $n$) if and only if the Fano manifold is
$\mathbb{P}^n$ (resp. a quadric hypersurface in
$\mathbb{P}^{n+1}$). Based on this simple observation, we are able
to obtain the following:
\begin{corollary}
Let $X$ be a Fano manifold of dimension $n\geq 3$. If the Fano
index of $X$ is at least $3$ and at most $n$, then the polarized
manifold $(X, -K_X)$ is slope stable with respect to any smooth
curve.
\end{corollary}
\begin{proof}
Because of the Fano index, the Fano manifold $X$ cannot be
$\mathbb{P}^n$. If a smooth curve $C$ destabilizes $(X, -K_X)$,
then $-K_X\cdot C=2$ or $1$ by Theorem~\ref{theorem:B}. However,
$X$ cannot have such a curve because of its Fano index.
\end{proof}

\begin{corollary}
Let $X$ be a Fano manifold of dimension $3$.
\begin{enumerate}
\item If the Fano index of $X$ is $2$, then the polarized
    manifold $(X, -K_X)$ is slope semistable with respect to
    any smooth curve.
\item If the Fano index of $X$ is $3$, i.e., the manifold $X$
    is a quadric hypersurface in $\mathbb{P}^4$, then the
    polarized manifold $(X, -K_X)$ is slope stable with
    respect to any smooth curve.
\end{enumerate}
\end{corollary}
\begin{proof}
Theorem~\ref{theorem:B} and the same argument as the previous
corollary give us an immediate proof.
\end{proof}

{\em Acknowledgements}.\label{ackref} Yongnam Lee would like to
thank Noboru Nakayama and Yuji Odaka for valuable comments at an
informal seminar talk in RIMS. The paper was worked out during his
visit to KIAS and RIMS in 2010 as an affiliate member and as a
visiting professor. He appreciates KIAS and RIMS for hospitality and
financial support. Jihun Park would like to thank Yuji Sano for
helpful conversations at Ise-Shima and Nagoya in Japan and at Pohang
in Korea where the ideas on Fano bundles occurred to him. The
authors would like to thank a referee for valuable comments to
modify the original version.


\end{document}